\documentclass[12pt]{amsart}

\usepackage{amssymb, amscd, txfonts}
\usepackage{graphicx}
\usepackage[all]{xy} 
\usepackage{comment} 
 
\numberwithin{equation}{section}

\newtheorem{theorem}{Theorem}[section]
\newtheorem{proposition}[theorem]{Proposition}
\newtheorem{lemma}[theorem]{Lemma}

\theoremstyle{definition}

\theoremstyle{remark}

\renewcommand{\hom}{\operatorname{Hom}}
\renewcommand{\ker}{\operatorname{Ker}}

\newcommand{\Sym}{\operatorname{Sym}}

\newcommand{\Z}{\mathbb{Z}}

\newcommand{\C}{\mathbb{C}}
\newcommand{\symC}{C^{(2)}}

\makeatletter
\def\subsection{\@startsection{subsection}{2}%
  \z@{.5\linespacing\@plus.7\linespacing}{.5\linespacing}%
  {\normalfont\bfseries}}
\makeatother

\begin{document}

%%%%%%% Title %%%%%%%%%%%%%%%%%%%%%%%%
\title[]{The mixed Hodge structure on the fundamental groups of the Collino surfaces}
\author[]{Daichi Arimatsu}
%\thanks{} 
\address{Department~of~Mathematics,~Institute~of~Science~Tokyo, Tokyo 152-8551, Japan}
\email{arimatsu.d.aa@m.titech.ac.jp}
%\subjclass[2020]{}
%\keywords{} 
%\dedicatory{}

\begin{abstract}
Collino proved that the fundamental group of a certain Zariski open set of the symmetric square of a hyperelliptic curve is isomorphic to the integral Heisenberg group.
We compute the mixed Hodge structure on this fundamental group, and show that the second extension class is expressed by the Abel-Jacobi invariant of the canonical class and the marked points of the hyperelliptic curve, together with a certain $\mathbb{F}_2$-linear map.
\end{abstract}

\maketitle

\section{Introduction}\label{sec: intro}

The theory of mixed Hodge structures on the fundamental group introduced by Hain \cite{Hain}
provides a framework for studying the interplay between the topology and Hodge theory of algebraic varieties. 
It has been observed that the extension classes for the weight filtration encode rich geometric information.
For instance, in the case of curves, the works of Hain \cite{Hain}, Pulte \cite{Pulte}, and Kaenders \cite{Kae}
have shown that these extension classes recover classical geometric invariants via the Abel-Jacobi map.

The purpose of this paper is to extend this series of investigations to certain open surfaces related to hyperelliptic curves.
Let $C$ be a hyperelliptic curve of genus $g \ge 2$, and let $\symC$ be the symmetric square of $C$. We choose a base point $b \in C$,
and consider the open surface
\[
V_b = C^{(2)} \setminus (L + C_b),
\]
where $L \cong \mathbb{P}^1$ is the hyperelliptic pencil and $C_b = C + b$ is the embedded $C$ with the base point $b$. 
Collino \cite{Co} proved that the fundamental group of $V_b$ is isomorphic to the integral Heisenberg group.
Thus, this provides a non-trivial but accessible example.
We call $V_b$ a \textit{Collino surface}.

We choose a base point $p+q$ of $V_b$, where $p,q \in C$. Hain's theory furnishes a mixed Hodge structure on the dual of $\Z \pi_1(V_b, p+q)/J^{s+1}$, where $J$ is the augmentation ideal of the group ring $\Z \pi_1(V_b, p+q)$.
In the case $s=2$, we obtain a mixed Hodge structure of the form
\begin{equation}\label{eq:intro-ext}
0 \to H^1(V_b,\Z) \to (J/J^3)^* \to K \to 0,
\end{equation}
where $H^1(V_b,\Z)$ is the weight $1$ part, and $K$ is the weight $2$ part, isomorphic to the kernel of the cup product map $H^1(V_b,\Z)^{\otimes 2} \to H^2(V_b,\Z)$.
Let $v_0 \in \Lambda^2 H^1(C,\Z)$ be the element corresponding to the cup product on $H^1(C,\Z)$.
We prove that $H^1(V_b,\Z) \cong H^1(C,\Z)$ and that $K$ contains
\begin{equation*}
K_0 \coloneqq \Sym^2 H^1(C,\Z) \oplus \Z v_0
\end{equation*}
as a sublattice of index $2$, with both summands primitive.
By restricting \eqref{eq:intro-ext} to $K_0$, we obtain a mixed Hodge structure which is an extension of $K_0$ by $H^1(C,\Z)$. We denote by $m_{p+q} \in \mathrm{Ext}_{\mathrm{MHS}}(K_0, H^1(C, \Z))$ its extension class.
By Carlson's isomorphism \cite{Car}, we have
\begin{equation} \label{ext}
\mathrm{Ext}_{\mathrm{MHS}}(K_0, H^1(C, \Z)) \cong JH \oplus \mathrm{Pic}^0(C),
\end{equation}
where $JH$ is the intermediate Jacobian of the weight $-1$ Hodge structure
\[
H \coloneqq \hom(\Sym^2 H^1(C,\Z), H^1(C,\Z)).
\]
Our main result is the following.
\begin{theorem}\label{thm: main}
Under the isomorphism \eqref{ext}, the extension class $m_{p+q}$ is given by
\[
m_{p+q} = (\alpha, \: p' + q' - p - q + (2g-2)b - K_C),
\]
where $p', q'$ are the conjugates of $p, q$ under the hyperelliptic involution, $K_C$ is the canonical divisor of $C$, and $\alpha$ is the $2$-torsion point of $JH$ corresponding to the $\mathbb{F}_2$-linear map
\[
\Sym^2H^1(C,\mathbb{F}_2) \to H^1(C,\mathbb{F}_2),
\]
that sends $v \otimes v \mapsto v$ for $v \in H^1(C,\mathbb{F}_2)$.
\end{theorem}

In this way, the information of the marked points $b$, $p$, and $q$ is explicitly reflected in the extension class $m_{p+q}$.
On the other hand, $\alpha$ is essentially the second extension class for $\pi_1(\symC, p+q)$ and is independent of the choice of base point. The dependence on the base point shows up only after removing the divisor $L + C_b$.

The proof of Theorem \ref{thm: main} uses Hain's formula \cite{Hain} which expresses the extension class by means of iterated integrals.
The first step is to determine the $JH$-component, and it turns out that this follows from the shuffle relation for iterated integrals.
To compute the $\mathrm{Pic}^0(C)$-component, we pull back the iterated integral to $C$ by $C \cong C_q \subset \symC$.
This reduces the problem to the calculation of an iterated integral on the two-punctured curve $C \setminus \{b, q'\}$.
The key step is to establish a higher reciprocity law in the two-punctured case.
This is an analogue of Kaenders' reciprocity law \cite{Kae} in the one-punctured case.
Finally, by combining our reciprocity law and the higher period relation of Kaenders \cite{Kae}, we arrive at our main result.

The rest of this paper is organized as follows. In Section 2, we recall Hain's theory and Collino's construction. 
In Section 3, we compute the second graded quotient. 
In Section 4, we prove Theorem \ref{thm: main}.

Throughout this paper, for a $\Z$-module $M$ of finite rank, $\Sym^2 M$ and $\Lambda^2 M$ stand for the sub $\Z$-modules of $M^{\otimes 2}$ generated by $x \otimes x$ ($x \in M$) and by $x \otimes y - y \otimes x$ ($x,y \in M$), respectively.

\section{The weight filtration}\label{sec: weight}

In this section, we recall Hain's theory and Collino's construction.
Let $C$ be a hyperelliptic curve of genus $g \ge 2$.
Let $\symC$ be the symmetric square of $C$, which parameterizes effective divisors of degree $2$ on $C$.
Let $\iota \colon C \to C$ denote the hyperelliptic involution.
Let $L \subset \symC$ be the hyperelliptic pencil, namely $L \coloneqq \{ x + \iota(x) \mid x \in C \}$. 
For a point $a \in C$, let $j_a \colon C \to \symC$ be the embedding defined by $x \mapsto x+a$. We denote the image of $j_a$ by $C_a$.

We choose a base point $b$ for $C$. 
We consider the Zariski open set
\[
V_b \coloneqq \symC \setminus (L+C_b).
\]
Collino \cite{Co} proved that the fundamental group of the open surface
$V_b$ is isomorphic to the integral Heisenberg group of genus $g$,
which we denote by $H(g)$. This group is generated by $2g+1$ elements $\{ a_i, b_i, \delta \}_{1 \le i \le g}$ subject to the relations
\[
[a_i, a_j] = [b_i, b_j] = 1, \quad [a_i, \delta] = [b_i, \delta] = 1, 
\quad [a_i, b_j] = \begin{cases}
\delta \quad (i = j) \\
1 \quad (i \neq j)
\end{cases}
\]
for all $i, j$. The center of $H(g)$ is generated by $\delta$, and $H(g)$ fits into the central extension
\begin{equation*}
0 \to \Z \delta \to H(g) \to \Z^{2g} \to 0.
\end{equation*}

In the Collino isomorphism $\pi_1(V_b) \cong H(g)$,
the generator $\delta$ corresponds to the class $\sigma_L$ of a meridian loop around $L$, and $\delta^{g-1}$ corresponds to the class $\sigma_{C_b}$ of a meridian loop around $C_b$. In particular, we have
\begin{equation}\label{eq: Collinoeq}
\sigma_L^{g-1} = \sigma_{C_b} \quad \in \pi_1(V_b).
\end{equation}

We choose a base point $p+q$ of $V_b$, where $p,q\in C$ with $p,q \neq b$ and $q' \neq p$. Let $\Z \pi_1(V_b, p+q)$ be the group ring of the fundamental group.
The augmentation ideal $J$ is defined as the kernel of the augmentation map $\Z \pi_1(V_b, p+q) \to \Z$.
Recall that there is a canonical isomorphism $J/J^2 \cong H_1(V_b, \Z)$. Dually, $(J/J^2)^* \cong H^1(V_b,\Z)$. 
Let $i \colon V_b \to \symC$ be the inclusion. We compute the Hodge structure on $H^1(V_b, \Z)$.

\begin{lemma}\label{prop: wht1}
The mixed Hodge structure on $H^1(V_b, \Z)$ is pure, and we have isomorphisms of pure Hodge structures
\[
H^1(C,\Z) \xleftarrow{j_b^*} H^1(\symC,\Z) \xrightarrow{i^*} H^1(V_b,\Z).
\]
\end{lemma}

\begin{proof}
Since $j_b^*$ and $i^*$ are morphisms of mixed Hodge structures, it suffices to show that they are isomorphisms.
The fact that $j_b^*$ is an isomorphism is a classical result due to Macdonald \cite{Mac}.
To show that $i^*$ is an isomorphism, we consider the Gysin sequence for the pair $(\symC, V_b)$:
\begin{equation}\label{Gysin}
0 \to  H^1(\symC,\Z) \xrightarrow{i^*} H^1(V_b,\Z) 
\xrightarrow{\mathrm{Res}}  H^0(L,\Z) \oplus H^0(C_b,\Z)
\xrightarrow{\mathrm{Gysin}}  H^2(\symC,\Z).
\end{equation}
Under the Abel--Jacobi map $\symC \to \mathrm{Pic}^2(C)$, $L$ is contracted to a point, while $C_b$ is mapped to a translation of $C$. Therefore the classes $[L]$ and $[C_b]$ are linearly independent in $H^2(\symC,\Z)$, and the Gysin map is injective.
Consequently, the residue map vanishes. Therefore $i^*$ is an isomorphism.
\end{proof}

We write $\pi_1 = \pi_1(V_b, p+q)$ in what follows. By Hain's theory \cite{Hain}, $\hom_{\Z}(\Z\pi_1/J^{s+1},\Z)$ for $s \ge 0$ has a mixed Hodge structure which is compatible with the inclusions
\[
\hom_{\Z}(\Z\pi_1/J^{s+1},\Z) \hookrightarrow \hom_{\Z}(\Z\pi_1/J^{s+2},\Z).
\]
This is defined via iterated integrals. To be more precise, we have an isomorphism
\begin{equation*}
H^0(\bar{B}_s(\mathcal{A}^{\bullet}(\symC \log (L + C_b)))) \xrightarrow{\sim} \hom_{\Z}(J/J^{s+1}, \C),
\end{equation*}
where the left-hand side is the $\C$-vector space of iterated integrals of length $\le s$ on $V_b$ that are homotopy functionals and have at most logarithmic growth at the boundary $L + C_b$.
The weight and Hodge filtrations are naturally defined on the left side.
In our case, since $H^1(V_b,\Z)$ is pure of weight $1$ by Lemma \ref{prop: wht1}, the description in \cite{Hain} section 5 tells us that the weight filtration $W_\bullet$ on $(J/J^{s+1})^*$ is given by
\begin{equation*}
W_l\, (J/J^{s+1})^* = (J/J^{l+1})^*, \quad 1 \le l \le s.
\end{equation*}

We are interested in the extension data at the second step.
Let $K$ be the kernel of the cup product:
\[
K \coloneqq \ker\,\!\left(H^1(V_b,\Z) \otimes H^1(V_b,\Z) \xrightarrow{\cup} H^2(V_b,\Z) \right).
\]
Since $\pi_1(V_b, p+q)^{\mathrm{ab}} \cong H(g)^{\mathrm{ab}} \cong \Z^{2g}$, we have $H_1(V_b,\Z) \cong \Z^{2g}$ and hence it is torsion-free. Therefore, by Hain~\cite{Hain}, $(J^2/J^3)^*$ is canonically isomorphic to $K$, and we have the short exact sequence
\begin{equation}\label{def: ext}
0 \to H^1(V_b,\Z) \to (J/J^3)^* \to K \to 0
\end{equation}
of mixed Hodge structures with $H^1(V_b,\Z)$ the weight $1$ graded quotient and $K$ the weight $2$ graded quotient.

\section{The second graded quotient}\label{sec: Gr2}
In this section, we compute the second graded quotient $K$.
Let $x_1, \dots, x_{2g}$ be a symplectic basis of $H^1(C,\Z)$, namely $(x_i,x_j)=0$ unless $j=i+g$ or $i=j+g$, and $(x_i,x_{g+i})=1$ for $1 \le i \le g$ with respect to the cup product on $C$.
Then $\Sym^2 H^1(C,\Z)$ has the $\Z$-basis
\begin{equation}\label{eq: sym2-basis}
\begin{array}{l}
 x_i \otimes x_i \quad (1 \le i \le 2g),\\
 x_i \otimes x_j + x_j \otimes x_i \quad (1 \le i < j \le 2g).
\end{array}
\end{equation}

Let $v_0$ denote the element of $H^1(C,\Z)^{\otimes 2}$ corresponding to the cup product on $H^1(C,\Z)$. Explicitly,
\[
v_0 = \sum_{i=1}^{g} x_i \wedge x_{g+i} = \sum_{i=1}^{g} (x_i \otimes x_{g+i} - x_{g+i} \otimes x_i).
\]
Recall from \cite{Mac} that the pullback by the embedding $j_a \colon C \hookrightarrow \symC$ induces an isomorphism on $H^1$. This is independent of $a \in C$. Then, for $\alpha \in H^1(C,\Z)$, we write $\alpha' \coloneqq (j_a^*)^{-1}\alpha \in H^1(\symC,\Z)$. We set
\[
v_0' \coloneqq \sum_{i=1}^{g} x_i' \wedge x_{g+i}' \in H^1(\symC,\Z)^{\otimes 2},
\]
and
\[
K_0 \coloneqq \Sym^2 H^1(V_b,\Z) \oplus \Z i^*v_0' \subset H^1(V_b,\Z)^{\otimes 2}.
\]
Here $i \colon V_b \hookrightarrow \symC$ is the inclusion map.

\begin{proposition}\label{Prop: gr2}
The lattice $K_0$ is a sublattice of $K$ of index $2$, and both summands
$\Sym^2 H^1(V_b,\Z)$ and $\Z i^*v_0'$ are primitive in $K$.
\end{proposition}

For the proof of this proposition, we need the following computation. We denote by $\varphi \colon H^1(\symC,\Z)^{\otimes 2} \to H^2(\symC,\Z)$ the cup product map.

\begin{lemma}\label{prop: cal of cup}
We have
\begin{equation}\label{eq: cup of v_0}
\varphi(v_0') = (2g-2)[C] + 2[L] \: \in H^2(\symC,\Z),
\end{equation}
where $[C]$ is the class of the curves $C_a \subset \symC$.
In particular, we have $i^*\varphi(v_0') = 0$.
\end{lemma}

\begin{proof}
The last sentence follows from \eqref{eq: cup of v_0} because $[L]$ and $[C]$ vanish on $V_b$.
Therefore, it suffices to prove \eqref{eq: cup of v_0}.
Let $\pi \colon C^2 \to \symC$ denote the natural projection, and let $p_1, p_2 \colon C^2 \to C$ denote the first and second projections, respectively.
Since $H^2(\symC,\Z)$ is torsion-free (\cite{Mac}), the pullback by $\pi$ on $H^2$ is injective. Hence it is enough to prove \eqref{eq: cup of v_0} after pulling it back to $H^2(C^2,\Z)$.
First, we compute the pullback of $\varphi(v_0')$.
Since $\pi^*\alpha' = p_1^*\alpha + p_2^*\alpha$, we obtain
\begin{equation}\label{eq1}
\begin{split}
\pi^*\varphi(v_0') &= 2\sum_{i=1}^{g} (p_1^*x_i+p_2^*x_i) \cup (p_1^*x_{g+i}+p_2^*x_{g+i}) \\
&= 2\sum_{i=1}^{g} \biggl( p_1^*(x_i\cup x_{g+i}) + p_2^*(x_i\cup x_{g+i}) + p_1^*x_i\cup p_2^*x_{g+i} + p_2^*x_i\cup p_1^*x_{g+i} \biggr) \\
&= 2g(p_1^*[\mathrm{pt}] + p_2^*[\mathrm{pt}])+ 2\sum_{i=1}^{g} \biggl( p_1^*x_i\cup p_2^*x_{g+i} + p_2^*x_i\cup p_1^*x_{g+i} \biggr).
\end{split}
\end{equation}
where $[\mathrm{pt}]$ is the class of a point of $C$.
Next, we recall the formula for the class of the diagonal $[\Delta_{C^2}]$ in $H^2(C^2,\Z)$ (\cite{Bott-Tu}, Lemma 11.22):
\begin{equation}\label{eq: diag formula}
[\Delta_{C^2}] = p_1^*[\mathrm{pt}] + p_2^*[\mathrm{pt}] - \sum_{i=1}^{g} \biggl( p_1^*x_i\cup p_2^*x_{g+i} + p_2^*x_i\cup p_1^*x_{g+i} \biggr).
\end{equation}
By \eqref{eq: diag formula} and \eqref{eq1}, we obtain
\begin{align}
\pi^*\varphi(v_0') = (2g+2)(p_1^*[\mathrm{pt}] + p_2^*[\mathrm{pt}]) - 2[\Delta_{C^2}]. \label{eq: LHS result}
\end{align}

Now we compute the right-hand side of \eqref{eq: cup of v_0}.
Recall that $\pi^*[L]$ is the class of the graph $\Gamma_\iota$ of the involution $\iota$.
The class $[\Gamma_\iota]$ satisfies $[\Gamma_\iota] + [\Delta_{C^2}] = 2(p_1^*[\mathrm{pt}] + p_2^*[\mathrm{pt}])$. Thus,
\[
\pi^*(2[L]) = 2(2(p_1^*[\mathrm{pt}] + p_2^*[\mathrm{pt}]) - [\Delta_{C^2}]) = 4(p_1^*[\mathrm{pt}] + p_2^*[\mathrm{pt}]) - 2[\Delta_{C^2}].
\]
Adding multiples of $\pi^*[C] = p_1^*[\mathrm{pt}] + p_2^*[\mathrm{pt}]$, we obtain
\begin{align}
\phantom{=}\pi^*((2g-2)[C] + 2[L])
= (2g+2)(p_1^*[\mathrm{pt}] + p_2^*[\mathrm{pt}]) - 2[\Delta_{C^2}]. \label{eq: RHS result}
\end{align}
Comparing \eqref{eq: LHS result} and \eqref{eq: RHS result}, we obtain
\[
\pi^*\varphi(v_0') = \pi^*((2g-2)[C] + 2[L]).
\]
This proves \eqref{eq: cup of v_0}.
\end{proof}
We now prove Proposition \ref{Prop: gr2}.

\begin{proof}[(Proof of Proposition \ref{Prop: gr2})]
The inclusion $\Sym^2H^1(V_b,\Z) \subset K$ is immediate from the skew-symmetry of the cup product.
By Lemma \ref{prop: cal of cup}, we also have $\Z i^*v_0' \subset K$.
Hence $K_0 \subset K$.
In the following, we identify $K_0$ and $K$ as sublattices of $H^1(C,\Z)^{\otimes2}$ via the isomorphism $H^1(V_b,\Z) \cong H^1(C,\Z)$ in Lemma \ref{prop: wht1}.
Since the two summands of $K_0$ are clearly primitive in $H^1(C,\Z)^{\otimes2}$, they are primitive in $K$.

Next, we show that $K/K_0$ is finite.
In general, Sandling--Tahara \cite{Sand-Tahara} computed $J^s/J^{s+1}$ for groups whose lower central series has free abelian successive quotients.
If we apply their result to our integral Heisenberg group, we obtain a (non-canonical) isomorphism
\[
J^2/J^3 \cong ((J/J^2)^{\otimes2}/\Lambda^2(J/J^2)) \oplus \Z\delta.
\]
Hence $J^2/J^3$ has the same rank as $K_0$, so $K_0$ is of finite index in $K$.

Finally, we prove $K/K_0 \cong \Z/2\Z$. We set
\[
u\coloneqq \sum_{i=1}^{g}x_i\otimes x_{g+i}.
\]
Modulo $\Sym^2 H^1(C,\Z)$, we have
\[
2u
=\sum_{i=1}^{g}\bigl(x_i\otimes x_{g+i}+x_{g+i}\otimes x_i\bigr)
+\sum_{i=1}^{g}\bigl(x_i\otimes x_{g+i}-x_{g+i}\otimes x_i\bigr)
\equiv v_0.
\]
This shows that $2\bar{u}=0$ in $K/K_0$, where $\bar{u}$ is the image of $u$ in $K/K_0$. Clearly, $\bar{u}\neq 0$. Moreover, a direct computation shows that the image of $u$ in $H^1(C,\Z)^{\otimes 2}/\Sym^2 H^1(C,\Z)$ is primitive. Hence the image of $u$ generates the rank $1$ lattice $K/\Sym^2 H^1(C,\Z)$, so $\bar{u}$ generates $K/K_0$.
This completes the proof of Proposition \ref{Prop: gr2}.
\end{proof}
\section{Extension class}\label{sec: ext}
We define the extension class 
\[
m_{p+q} \in \mathrm{Ext}_{\mathrm{MHS}}(K_0, H^1(V_b,\Z))
\]
as the pullback of the extension class of \eqref{def: ext} by $K_0 \hookrightarrow K$.
Here $p+q$ is the base point of the fundamental group of $V_b$.
Let $J\mathrm{Hom}(K_0, H^1(V_b,\mathbb{Z}))$ denote the intermediate Jacobian associated to the pure Hodge structure $\mathrm{Hom}(K_0, H^1(V_b,\mathbb{Z}))$ of weight $-1$.
By Carlson's isomorphism \cite{Car} and the isomorphism $H^1(V_b,\Z) \cong H^1(C,\Z)$ in Lemma \ref{prop: wht1}, we have
\[
\Psi \colon \mathrm{Ext}_{\mathrm{MHS}}(K_0, H^1(V_b,\Z)) \xrightarrow{\cong} J\mathrm{Hom}(K_0, H^1(V_b,\Z)) \xrightarrow{\cong} J\mathrm{Hom}(K_0, H^1(C,\Z)).
\]
By the definition of $K_0$, we have
\begin{align}\label{eq: decomp.of.J}
J\mathrm{Hom}(K_0, H^1(C,\Z)) \cong JH \oplus J\mathrm{Hom}(\Z v_0, H^1(C,\Z))
\cong JH \oplus \mathrm{Pic}^0(C),
\end{align}
where $JH$ is the intermediate Jacobian of the weight $-1$ Hodge structure
\[
H \coloneqq \hom(\Sym^2 H^1(C,\Z), H^1(C,\Z)).
\]
According to this decomposition, we write 
\[
\Psi(m_{p+q}) = (\psi_1, \psi_2)
\]
where $\psi_1 \in JH$ and $\psi_2 \in \mathrm{Pic}^0(C)$. 

Recall that the set of $2$-torsion points of $JH$ is identified with
\[
\mathrm{Hom}_{\mathbb{F}_2}(\Sym^2H^1(C,\mathbb{F}_2),H^1(C,\mathbb{F}_2)).
\]
Over $\mathbb{F}_2$, we have $\Lambda^2H^1(C,\mathbb{F}_2) \subset \Sym^2H^1(C,\mathbb{F}_2)$, where $\Lambda^2$ is as defined in p.~3. Moreover, the map
\begin{equation}\label{eq:f2-quad-map}
 H^1(C,\mathbb{F}_2)\to \Sym^2H^1(C,\mathbb{F}_2)/\Lambda^2H^1(C,\mathbb{F}_2), \quad v \mapsto v \otimes v,
\end{equation}
is $\mathbb{F}_2$-linear. By a direct computation, this is an isomorphism.
We denote by $\theta$ the composition of the inverse of \eqref{eq:f2-quad-map} and the projection
\[
\Sym^2H^1(C,\mathbb{F}_2)\twoheadrightarrow \Sym^2H^1(C,\mathbb{F}_2)/\Lambda^2H^1(C,\mathbb{F}_2).
\]
Thus
\[
\theta \in \hom_{\mathbb{F}_2}(\Sym^2 H^1(C,\mathbb{F}_2), H^1(C,\mathbb{F}_2)).
\]
Then Theorem \ref{thm: main} can be reformulated as follows.

\begin{theorem}\label{thm: main'}
The following holds.
\begin{enumerate}
\item[(1)] $\psi_1$ is the $2$-torsion point in $JH$ corresponding to $\theta$.
\item[(2)] We have 
\[
\psi_2 = p' + q' - p - q + (2g-2)b - K_C,
\]
where $p', q'$ are the conjugates of $p, q$ under the hyperelliptic involution, and $K_C$ is the canonical divisor of $C$.
\end{enumerate}
\end{theorem}

The rest of this section is devoted to the proof of Theorem \ref{thm: main'}.
In \S \ref{sec: ext-iterated-integral}, we recall Hain's formula \cite{Hain} expressing the extension class in terms of Chen's iterated integrals.
In \S \ref{sec: proof1}, we prove Theorem \ref{thm: main'} (1).
In \S \ref{sec: proof2}, we reduce the proof of Theorem \ref{thm: main'} (2) to a higher reciprocity law on two-punctured curves (Proposition \ref{prop: rec.law}).
Finally, in \S \ref{sec: proof3}, we prove Proposition \ref{prop: rec.law}.

\subsection{Extension class and iterated integrals}\label{sec: ext-iterated-integral}
In this subsection, we recall Hain's formula \cite{Hain} expressing the extension class in terms of Chen's iterated integrals (in the present situation).
We consider the embedding
\[
j_q \colon C \setminus \{b, q'\} \hookrightarrow V_b, \quad x \mapsto x+q.
\]
Choose loops $\gamma_1, \dots, \gamma_{2g}$ in $C \setminus \{b, q'\}$ which represent a symplectic basis of
$H_1(C,\Z)$ such that $\int_{\gamma_k} x_i = \delta_{ik}$.
We denote the homology class of each loop $\gamma_k$ by the same symbol $\gamma_k$.
Then the pushforwards $(j_q)_*\gamma_1, \dots, (j_q)_*\gamma_{2g}$ form a basis of $H_1(V_b,\Z)$.
The point $\Psi(m_{p+q}) \in J\mathrm{Hom}(K_0,H^1(C,\Z))$ is represented by the linear map
\[
\tilde{\Psi} \colon K_0 \otimes \C \to H^1(C,\C)
\]
defined as follows.
Let $\mu = \sum_{i,j} [w_i]\otimes[w'_j]$ be an element of $K_0 \otimes \C$, where $w_i,w'_j \in \mathcal{A}^1(\symC \log (L+C_b))$ are closed $1$-forms.
We take an element $\eta$ of $\mathcal{A}^1(\symC \log (L+C_b))$ such that
\[
\sum w_i \wedge w'_j + d\eta = 0.
\]
Then $\tilde{\Psi}(\mu) \in H^1(C,\C)$ is given by
\begin{equation}\label{eq:Psi-mpq}
\tilde{\Psi}(\mu)
= \sum_{k=1}^{2g} \left( \int_{(j_q)_* \gamma_k} \sum w_i w'_j + \eta \right) x_k \: \in H^1(C,\C).
\end{equation}
According to the decomposition \eqref{eq: decomp.of.J}, we write
\[
\tilde{\Psi} = (\tilde{\psi}_1,\tilde{\psi}_2),
\]
where $\tilde{\psi}_1 \in \hom_{\C}(\Sym^2 H^1(C,\C), H^1(C,\C))$ and $\tilde{\psi}_2 \in H^1(C,\C)$.

We can pull back the iterated integral \eqref{eq:Psi-mpq} via $j_q$:
\begin{equation}\label{eq: pullback.iter.int}
\int_{(j_q)_* \gamma_k} \sum w_i w'_j + \eta
= \int_{\gamma_k} \left( \sum j_q^* w_i \, j_q^* w'_j + j_q^*\eta \right).
\end{equation}
This reduces the calculation of $\Psi(m_{p+q})$ to that of an iterated integral on the two-punctured curve $C\setminus\{b,q'\}$.

\subsection{Proof of Theorem \ref{thm: main'} (1)}\label{sec: proof1}
{\pushQED{\qed}
In this subsection, we prove Theorem \ref{thm: main'} (1). In view of the definition of $\theta$, it suffices to show that for all $v \in H^1(C,\Z)$,
\begin{equation}\label{eq:psi1-vv}
\tilde{\psi}_1(v \otimes v) = \frac{1}{2}\,v.
\end{equation}
By the basis description \eqref{eq: sym2-basis} of $\Sym^2 H^1(C,\Z)$, it suffices to show
\begin{equation}\label{eq:psi1-basis}
\begin{aligned}
\tilde{\psi}_1(x_i \otimes x_i) &= \frac{1}{2}x_i,\\
\tilde{\psi}_1(x_i \otimes x_j + x_j \otimes x_i) &= 0 \quad (i \ne j).
\end{aligned}
\end{equation}

In general, for an element of $\Sym^2 H^1(V_b,\Z)$ of the form $[w] \otimes [w'] + [w'] \otimes [w]$, where $w,w'$ are closed $1$-forms on $V_b$,
we have $w \wedge w' + w' \wedge w= 0$, so the $\eta$-term in \eqref{eq:Psi-mpq} is $0$.
By \eqref{eq:Psi-mpq}, we have
\[
\tilde{\psi}_1([w] \otimes [w'] + [w'] \otimes [w]) = \sum_{k=1}^{2g} \left( \int_{(j_q)_*\gamma_k} w w'+ w' w \right) x_k.
\]
We choose $w,w'$ so that $[j_q^*w] = x_i$ and $[j_q^*w'] = x_j$. 
By abuse of notation, we write $x_i = j_q^*w$ and $x_j = j_q^*w'$.
By \eqref{eq: pullback.iter.int}, we obtain
\begin{align*}
\tilde{\psi}_1(x_i \otimes x_j + x_j \otimes x_i)
&= \sum_{k=1}^{2g} \left( \int_{\gamma_k} x_i x_j + x_j x_i \right) x_k \\
&= \sum_{k=1}^{2g} \left( \int_{\gamma_k} x_i \cdot \int_{\gamma_k} x_j \right) x_k \\
&= \delta_{ij}\,x_i.
\end{align*}
Here, the second equality follows from the shuffle relation for iterated integrals.
Thus \eqref{eq:psi1-basis} follows.
This completes the proof of Theorem \ref{thm: main'} (1).
\popQED}

\subsection{Reduction of the proof of Theorem \ref{thm: main'} (2)} \label{sec: proof2}
In this subsection, we reduce the proof of Theorem \ref{thm: main'} (2) to a higher reciprocity law on two-punctured curves (Proposition \ref{prop: rec.law}).
This is done in two steps. First, we reformulate Theorem \ref{thm: main'} (2) in terms of the Abel--Jacobi map (Theorem \ref{thm: main''}). Next, we reduce Theorem \ref{thm: main''} to Proposition \ref{prop: rec.law}.

Let $\omega_1, \dots, \omega_g$ be a basis of holomorphic $1$-forms on $C$ satisfying
\[
\int_{\gamma_{\nu}} \omega_i = \delta_{i \nu} \quad (1 \leq i, \nu \leq g).
\]
We set
\[
\Omega = (I, Z) = \left( \int_{\gamma_k} \omega_j \right)_{\substack{1\leq j \leq g \\ 1\leq k \leq 2g}}.
\]
By Riemann's bilinear relations, $Z$ is a symmetric $g \times g$ matrix with positive-definite imaginary part.
We define
\[
\mathrm{Jac}(C) \coloneqq \C^g/\Omega\Z^{2g}.
\]
Let $u \colon \mathrm{Pic}^0(C) \to \mathrm{Jac}(C)$ denote the Abel--Jacobi map. For an element $[D] = \left[\sum_k n_k p_k\right]$ of $\mathrm{Pic}^0(C)$, where $p_k \in C$ and $\sum_k n_k = 0$, we have
\begin{equation}\label{eq:AJmap-u}
u([D]) = \left(\sum_k n_k\int_p^{p_k}\omega_i\right)_{i=1,\dots,g} \bmod \Omega\Z^{2g},
\end{equation}
where $p$ is a component of the chosen base point $p+q$ of $V_b$.
Let $\kappa_p \in \mathrm{Jac}(C)$ denote the Riemann constant with respect to the base point $p$.
Theorem \ref{thm: main'} (2) can be reformulated as follows.
\begin{theorem}\label{thm: main''}
Under the Abel--Jacobi map, the image of the extension class is given by
\[
u(\psi_2) = \left( 2\int_p^{q'} \omega_i + (2g-2)\int_p^b \omega_i \right)_i + 2\kappa_p.
\]
\end{theorem}

We explain how Theorem \ref{thm: main'} (2) follows from Theorem \ref{thm: main''}.
\begin{proof}[(Proof of Theorem \ref{thm: main''} $\Rightarrow$ Theorem \ref{thm: main'} (2))]
Since $p+p' \sim q+q'$, we have $p'-q \sim q'-p$, and therefore
\[
p' + q' - p - q \sim 2(q'-p).
\]
By \eqref{eq:AJmap-u}, we have
\[
u\bigl(p' + q' - p - q + (2g-2)b - (2g-2)p\bigr)
= \left( 2\int_p^{q'} \omega_i + (2g-2)\int_p^b \omega_i \right)_i.
\]
Since we have
\begin{equation}\label{eq:riemann-constant}
u((2g-2)p - K_C) = 2\kappa_p
\end{equation}
(see, for example, \cite[p.~340]{GH}), we obtain
\[
u\bigl(p' + q' - p - q + (2g-2)b - K_C\bigr)
= \left( 2\int_p^{q'} \omega_i + (2g-2)\int_p^b \omega_i \right)_i + 2\kappa_p.
\]
Hence, Theorem \ref{thm: main''} implies
\[
u(\psi_2) = u\bigl(p' + q' - p - q + (2g-2)b - K_C\bigr).
\]
Since $u$ is an isomorphism, Theorem \ref{thm: main'} (2) follows.
\end{proof}

In the rest of this subsection, we reduce Theorem \ref{thm: main''} to Proposition \ref{prop: rec.law}.
We begin by expressing $\psi_2$ in terms of iterated integrals.
Recall from \S 3 that $x_i'$ is an element of $H^1(\symC,\Z)$ such that $j_a^*x_i' = x_i$. We choose a closed $1$-form on $\symC$ representing this class, and denote it again by $x_i'$.
Let $\eta_b$ be a $C^\infty$ $1$-form on $\symC$ with logarithmic singularities along $L + C_b$
such that
\begin{equation} \label{eq: eta}
2\sum_{i = 1}^g x_i' \wedge x_{g+i}' + d\eta_b = 0.
\end{equation}
Here the first term is a closed $2$-form on $\symC$ representing the class $\varphi(v_0')$ defined in Section~\ref{sec: Gr2}.
Set $\xi_{b,q} \coloneqq j_q^*\eta_b \in \mathcal{A}^1(C \log(b+q'))$. By \eqref{eq:Psi-mpq} and \eqref{eq: pullback.iter.int}, we have
\begin{equation*}
\tilde{\psi}_2
= \sum_{k=1}^{2g} \left( \int_{\gamma_k} \left( \sum_{i=1}^{g} (x_i x_{g+i} - x_{g+i} x_i) + \xi_{b,q} \right) \right)x_k.
\end{equation*}
We denote this iterated integral by
\begin{equation}\label{eq:psi2-concrete}
\tilde{\psi}_2 = \sum_{k=1}^{2g} \left( \int_{\gamma_k} v_0 + \xi_{b,q} \right) x_k.
\end{equation}

Next, for two functions $F, G \colon \pi_1(C\setminus\{b,q'\},p) \to \C$, we define
\[
\Pi(F; G) \coloneqq \sum_{\nu=1}^{g} \left( F(\gamma_{\nu}) G(\gamma_{g+\nu}) - F(\gamma_{g+\nu}) G(\gamma_{\nu}) \right).
\]
This is an analogue of Kaenders' symbol \cite{Kae}, who introduced this symbol for one-punctured curves.
With this notation, we can state a higher reciprocity law on $C \setminus \{b, q'\}$.

\begin{proposition}\label{prop: rec.law}
For a holomorphic $1$-form $\omega$ on $C$, the following equality holds modulo $\Z\int_{\gamma_1}\omega + \cdots + \Z\int_{\gamma_{2g}}\omega$:
\begin{equation}\label{eq: rec.law}
\begin{split}
 & \sum_{\nu=1}^{g} \left( \int_{\gamma_{\nu}}{\omega} \cdot \int_{\gamma_{g+\nu}}{ (v_0+\xi_{b,q})}
-\int_{\gamma_{g+\nu}} {\omega} \cdot \int_{\gamma_{\nu}}{(v_0+ \xi_{b,q})} \right) \\
&\equiv 2\int_p^{q'} {\omega} + (2g-2)\int_p^b {\omega} \\
&\quad + \sum_{j,k=1}^{g} a_{jk} \Biggl\{ -2\Pi \left( \int \omega \omega_j ; \int \bar{\omega}_k \right) + 2 \Pi\left( \int{\omega} \int{\bar{\omega}_k} ; 
\int{\omega_j} \right) \\
&\quad +\Pi\left( \int \omega ; \int \omega_j \int \bar{\omega}_k \right) \Biggr\},
\end{split}
\end{equation}
where $(a_{jk})_{j,k} = (\bar{Z}-Z)^{-1}$.
\end{proposition}

We now explain how Theorem \ref{thm: main''} follows from Proposition \ref{prop: rec.law}.
\begin{proof}[(Proof of Proposition \ref{prop: rec.law} $\Rightarrow$ Theorem \ref{thm: main''})]
We first rewrite $u(\psi_2)$ in a form to which Proposition \ref{prop: rec.law} can be applied.
Let $\mathrm{PD}\colon H^1(C,\mathbb{Z}) \xrightarrow{\sim} H_1(C,\mathbb{Z})$ be the Poincar\'e duality isomorphism. It induces an isomorphism of abelian varieties
\[
\overline{\mathrm{PD}}\colon H^1(C,\mathbb{C})/(F^1H^1(C,\mathbb{C})+H^1(C,\mathbb{Z})) \xrightarrow{\sim} H_1(C,\mathbb{C})/(F^0H_1(C,\mathbb{C})+H_1(C,\mathbb{Z})).
\]
Via the standard isomorphism
\[
\mathrm{Pic}^0(C) \xrightarrow{\sim} H^1(C,\mathbb{C})/(F^1+H^1(C,\mathbb{Z})),
\]
the map $u$ can be expressed as follows (see, e.g., \cite{Kae}):
\[
u\colon H^1(C,\mathbb{C})/(F^1+H^1(C,\mathbb{Z})) \to \mathbb{C}^g/\Omega\mathbb{Z}^{2g}, \qquad
[\zeta] \mapsto \left(\int_{\overline{\mathrm{PD}}([\zeta])}\omega_i\right)_i.
\]
For the symplectic basis chosen above, we have
\[
\mathrm{PD}(x_{\nu}) = -\gamma_{g+\nu}, \qquad \mathrm{PD}(x_{g+\nu}) = \gamma_{\nu} \qquad (1 \le \nu \le g).
\]
Hence,
\[
u(x_k) =
\begin{cases}
-\left(\int_{\gamma_{g+k}} \omega_i\right)_i & (1 \le k \le g), \\
\left(\int_{\gamma_{k-g}} \omega_i\right)_i & (g+1 \le k \le 2g).
\end{cases}
\]
Combining this with \eqref{eq:psi2-concrete}, we obtain
\begin{equation}\label{eq: ext.hom.}
u(\psi_2) = \left( \sum_{\nu=1}^{g} \left( \int_{\gamma_{\nu}} \omega_i \cdot \int_{\gamma_{g+\nu}} (v_0+\xi_{b,q}) - \int_{\gamma_{g+\nu}} \omega_i \cdot \int_{\gamma_{\nu}} (v_0+ \xi_{b,q}) \right) \right)_i.
\end{equation}
On the other hand, the functions on $\pi_1(C \setminus \{b,q'\})$ appearing on the right-hand side of \eqref{eq: rec.law} actually descend to functions on $\pi_1(C)$, since the $1$-forms there are defined on $C$. Therefore we can apply the higher period relations of Kaenders \cite{Kae}, which say that
\begin{equation}\label{eq: high.pr.rel.}
\begin{split}
\Biggl( \sum_{j,k=1}^{g} a_{jk} \Biggl\{ -2\Pi \left( \int \omega_i \omega_j ; \int \bar{\omega}_k \right) 
+ 2 \Pi\left( \int{\omega_i} \int{\bar{\omega}_k} ; \int{\omega_j} \right) \\
+\Pi\left( \int \omega_i ; \int \omega_j \int \bar{\omega}_k \right) \Biggr\} \Biggr)_i \equiv 2\kappa_p \quad \pmod{\Omega \Z^{2g}}.
\end{split}
\end{equation}
Applying \eqref{eq: rec.law} to each component of \eqref{eq: ext.hom.} and combining it with \eqref{eq: high.pr.rel.}, we obtain Theorem \ref{thm: main''}.
\end{proof}

Thus, Theorem \ref{thm: main'} (2) is reduced to Proposition \ref{prop: rec.law}.

\subsection{Proof of Proposition \ref{prop: rec.law}} \label{sec: proof3}
{\pushQED{\qed}
In this subsection, we prove Proposition \ref{prop: rec.law}.
This is done in several steps.
We begin by defining the iterated integral $\int I$ of length $\le 3$ on $C \setminus \{b,q'\}$ by
\[
\int I \coloneqq \int \sum_{j,k=1}^{g} 2a_{jk}\omega \omega_j \bar{\omega}_k + \omega \xi_{b,q}.
\]
\begin{lemma}\label{lem:I-homotopy-functional}
The iterated integral $\int I$ is a homotopy functional on $C \setminus \{b,q'\}$.
\end{lemma}

\begin{proof}
By a standard computation, the element $v_0$ can be written as
\begin{equation}\label{eq:v0-decomp}
v_0 = \sum_{i=1}^{g} (x_i \otimes x_{g+i} - x_{g+i} \otimes x_i) = \sum_{j,k=1}^{g} (a_{jk} \omega_j \otimes \bar{\omega}_k + \bar{a}_{jk} \bar{\omega}_j \otimes \omega_k).
\end{equation}
Therefore, pulling back \eqref{eq: eta} via $j_q$ to $C \setminus \{b,q'\}$, we obtain
\begin{equation} \label{eq: xi}
\sum_{i=1}^{g} 2 x_i \wedge x_{g+i} + d\xi_{b,q} = \sum_{j,k=1}^{g} 2a_{jk}\omega_j \wedge \bar{\omega}_k + d\xi_{b,q} = 0.
\end{equation}
By \cite[Remark~4.5]{Hain}, this implies that $\int I$ is a homotopy functional.
\end{proof}

Next, we consider the elements $c_i \coloneqq (\gamma_i-1)$, $d_b \coloneqq (\delta_{b}-1)$, and $d_{q'} \coloneqq (\delta_{q'}-1)$ of the augmentation ideal $J$ of $\Z\pi_1(C\setminus \{b,q'\},p)$, where $\delta_b$ and $\delta_{q'}$ denote meridian loops based at $p$ around $b$ and $q'$, respectively.
\begin{lemma}\label{lem:modJ4-relation}
In $\Z\pi_1(C\setminus \{b,q'\},p)$, we have
\begin{equation}\label{eq: modJ^4}
\begin{aligned}
\sum_{\nu=1}^{g} \bigl\{ &(c_{\nu}c_{g+\nu} - c_{g+\nu}c_{\nu}) + (c_{g+\nu}c_{\nu}c_{g+\nu} - c_{\nu}c_{g+\nu}c_{\nu}) \\
\qquad &- (c_{\nu}c_{g+\nu}c_{g+\nu} - c_{g+\nu}c_{\nu}c_{\nu}) \bigr\} 
\equiv d_b + d_{q'} + d_bd_{q'} \pmod{J^4}.
\end{aligned}
\end{equation}
\end{lemma}

\begin{proof}
Consider the following relation in $\pi_1(C \setminus \{b, q'\}, p)$:
\[
[\gamma_1, \gamma_{g+1}] \cdots [\gamma_g, \gamma_{2g}] = \delta_{b} \delta_{q'}
\]
(cf.~\cite[(1.4)]{Kae}).
By \cite[\S 1.3.1]{Kae}, we have
\begin{multline*}
[\gamma_1, \gamma_{g+1}] \cdots [\gamma_g, \gamma_{2g}] -1 \\
\equiv \sum_{\nu=1}^{g} \bigl\{ (c_{\nu}c_{g+\nu} - c_{g+\nu}c_{\nu}) + (c_{g+\nu}c_{\nu}c_{g+\nu} - c_{\nu}c_{g+\nu}c_{\nu}) - (c_{\nu}c_{g+\nu}c_{g+\nu} - c_{g+\nu}c_{\nu}c_{\nu}) \bigr\} \pmod{J^4}.
\end{multline*}
On the other hand, we compute
\[
\delta_b\delta_{q'} - 1 =(1+d_b)(1+d_{q'}) - 1 \equiv d_b+d_{q'}+d_bd_{q'} \pmod{J^4}.
\]
Comparing these two expressions proves \eqref{eq: modJ^4}.
\end{proof}

Since the iterated integral $I$ has length $\le 3$, pairing with $I$ gives the same value on both sides of \eqref{eq: modJ^4}.
We evaluate the two sides separately.

First, we evaluate $\int I$ on the left-hand side of \eqref{eq: modJ^4}.

\begin{lemma}[cf.~\cite{Kae}]\label{lem:LHSofmodJ^4}
We have
\begin{equation}\label{eq:pairing-LHS-modJ4}
\begin{aligned}
&\left\langle I, \text{LHS of \eqref{eq: modJ^4}}\right\rangle \\
&= \sum_{\nu=1}^{g} \left( \int_{\gamma_{\nu}}{\omega} \cdot \int_{\gamma_{g+\nu}}{ (v_0+\xi_{b,q})}
-\int_{\gamma_{g+\nu}} {\omega} \cdot \int_{\gamma_{\nu}}{(v_0+ \xi_{b,q})} \right) \\
&\quad - \sum_{j,k=1}^{g} a_{jk} \Biggl\{ -2\Pi \left( \int \omega \omega_j ; \int \bar{\omega}_k \right) + 2 \Pi\left( \int{\omega} \int{\bar{\omega}_k} ; \int{\omega_j} \right)\\
&\qquad + \Pi\left( \int \omega ; \int \omega_j \int \bar{\omega}_k \right) \Biggr\}.
\end{aligned}
\end{equation}
\end{lemma}

\begin{proof}
This is proved by the same calculation as in the first paragraph of the proof of Theorem~1.4 of Kaenders~\cite{Kae}. 
While Kaenders considers one-punctured curves, 
the same argument applies to the two-punctured curve $C \setminus \{b,q'\}$. 
The only necessary modification is to replace the form $\mu_q$ in~\cite{Kae} by $\xi_{b,q}$ here.
\end{proof}

Next, we evaluate $\int I$ on the right-hand side of \eqref{eq: modJ^4}.
\begin{lemma}\label{lem:dbdq-I-vanish}
We have
\[
\int_{d_bd_{q'}} I =0.
\]
\end{lemma}

\begin{proof}
For the length $3$ part, by Proposition (2.9) in \cite{Hain}, we have
\[
\int_{d_bd_{q'}} \omega\,\omega_j\,\bar{\omega}_k
= \int_{d_b}\omega\,\omega_j\int_{d_{q'}}\bar{\omega}_k
+ \int_{d_b}\omega\int_{d_{q'}}\omega_j\,\bar{\omega}_k.
\]
Since $\omega$, $\omega_j$, and $\bar{\omega}_k$ are defined on $C$, their integrals over the meridian loops vanish, so in particular
\[
\int_{d_b}\omega = \int_{d_{q'}}\bar{\omega}_k =0.
\]
Hence the length $3$ part vanishes on $d_bd_{q'}$. For the length $2$ part, the same formula gives
\[
\int_{d_bd_{q'}} \omega\,\xi_{b,q}
= \int_{d_b}\omega\int_{d_{q'}}\xi_{b,q} =0.
\]
Therefore $\int_{d_bd_{q'}} I=0$.
\end{proof}

\begin{lemma}\label{lem:pairing-RHS-decomp}
We have the following identity modulo $\Z\int_{\gamma_1}\omega + \cdots + \Z\int_{\gamma_{2g}}\omega$:
\begin{equation}\label{eq:pairing-RHS-decomp}
\int_{d_b + d_{q'}} I
\equiv 2\pi\sqrt{-1}\mathrm{Res}_b(\xi_{b,q})\int_p^b \omega
+ 2\pi\sqrt{-1}\mathrm{Res}_{q'}(\xi_{b,q})\int_p^{q'} \omega.
\end{equation}
\end{lemma}

\begin{proof}
We choose local coordinates $(U_b,z_b)$ and $(U_{q'},z_{q'})$ of $C$ that contain representatives of $\delta_b$ and $\delta_{q'}$, respectively.
We can write locally
\begin{equation}\label{eq:local-xi-decomp}
\begin{aligned}
\xi_{b,q} &= \alpha_b + \mathrm{Res}_b(\xi_{b,q})\frac{dz_b}{z_b-z_b(b)}\quad\text{on }U_b, \\
\xi_{b,q} &= \alpha_{q'} + \mathrm{Res}_{q'}(\xi_{b,q})\frac{dz_{q'}}{z_{q'}-z_{q'}(q')}\quad\text{on }U_{q'}.
\end{aligned}
\end{equation}
where $\alpha_b$ and $\alpha_{q'}$ are $C^\infty$ $1$-forms on $U_b$ and $U_{q'}$, respectively.
Since the second term is $d$-closed, we see that the iterated integral
\[
\int \sum_{j,k=1}^{g} 2a_{jk}\,\omega\,\omega_j\,\bar{\omega}_k + \omega\alpha_b
\]
is a homotopy functional by the same calculation as the proof of Lemma \ref{lem:I-homotopy-functional}. 
Since this is defined over $U_b$, we have
\begin{equation}\label{eq:db-alpha-b-vanish}
\int_{d_b}\sum_{j,k=1}^{g} 2a_{jk}\,\omega\,\omega_j\,\bar{\omega}_k + \omega\alpha_b=0.
\end{equation}
Therefore, only the logarithmic term remains; namely,
\[
\int_{d_b} I
= \mathrm{Res}_b(\xi_{b,q})\int_{\delta_b} \omega\,\frac{dz_b}{z_b-z_b(b)}.
\]
Similarly, we have
\[
\int_{d_{q'}} I
= \mathrm{Res}_{q'}(\xi_{b,q})\int_{\delta_{q'}} \omega\,\frac{dz_{q'}}{z_{q'}-z_{q'}(q')}.
\]
By standard computations of iterated integrals on $U_b$ and $U_{q'}$ (cf.~\cite[p.~1276]{Kae}), we also have
\begin{equation}\label{eq:local-int-identities}
\begin{aligned}
\int_{\delta_b} \omega \,\frac{dz_b}{z_b-z_b(b)} &= 2\pi\sqrt{-1}\int_p^b \omega, \\
\int_{\delta_{q'}} \omega \, \frac{dz_{q'}}{z_{q'}-z_{q'}(q')} &= 2\pi\sqrt{-1}\int_p^{q'} \omega,
\end{aligned}
\end{equation}
where the integrals $\int_p^b \omega$ and $\int_p^{q'} \omega$ are well defined modulo
$\Z\int_{\gamma_1}\omega + \cdots + \Z\int_{\gamma_{2g}}\omega$.
This proves \eqref{eq:pairing-RHS-decomp}.
\end{proof}

Finally, we calculate the residues of $\xi_{b,q}$ at $b$ and $q'$.
\begin{lemma}\label{lem:res-explicit}
We have
\begin{equation}\label{eq:res-explicit}
\mathrm{Res}_b(\xi_{b,q}) = \frac{2g-2}{2\pi \sqrt{-1}}, \quad \mathrm{Res}_{q'}(\xi_{b,q}) = \frac{2}{2\pi \sqrt{-1}}.
\end{equation}
\end{lemma}

\begin{proof}
We first compute their sum:
\begin{eqnarray}\label{eq: res_sum}
2\pi\sqrt{-1} \left( \mathrm{Res}_b(\xi_{b,q}) + \mathrm{Res}_{q'}(\xi_{b,q}) \right) 
&=& \int_{\delta_b}\xi_{b,q} + \int_{\delta_{q'}}\xi_{b,q} \nonumber \\
&=& \int_C \sum_{i=1}^g 2 x_i \wedge x_{g+i} = 2g. 
\end{eqnarray}
In the second equallity we used \eqref{eq: xi} and  Stokes' theorem.
Next, we compute their ratio.
We have
\[
\int_{\delta_b} \sum_{i=1}^g 2 x_i x_{g+i} + \alpha_b = 0
\]
as in \eqref{eq:db-alpha-b-vanish}. Hence by \eqref{eq:local-xi-decomp},
\[
2\pi \sqrt{-1} \mathrm{Res}_b(\xi_{b,q}) = \int_{\delta_b} \sum_{i=1}^g 2 x_i x_{g+i} + \xi_{b,q}.
\]
By \eqref{eq: pullback.iter.int} and $(j_q)_*\delta_b = \sigma_{C_b}$ in $\pi_1(V_b, p+q)$, this gives
\begin{equation}\label{eq:res-b}
2\pi \sqrt{-1} \mathrm{Res}_b(\xi_{b,q}) = \int_{\sigma_{C_b}} \sum_{i=1}^g 2 x'_i x'_{g+i} + \eta_b.
\end{equation}
We denote this iterated integral of length $\le 2$ by $\int_{\sigma_{C_b}} I'$.
Similarly,
\begin{equation}\label{eq:res-qprime}
2\pi \sqrt{-1} \, \mathrm{Res}_{q'}(\xi_{b,q}) = \int_{\sigma_L} I'.
\end{equation}
Recall from \S 2 that the Heisenberg relation identifies $\sigma_L$ with the commutator $[a_i,b_i]$ in $\pi_1(V_b,p+q)$. Hence $\sigma_L - 1 \in J^2$, where $J$ denotes the augmentation ideal of $\Z\pi_1(V_b,p+q)$. Hence $(\sigma_L-1)^2 \in J^4 \subset J^3$.
Combining this with the Collino relation \eqref{eq: Collinoeq}, we obtain
\[
\sigma_{C_b} = (1 + (\sigma_L-1))^{g-1} \equiv 1 + (g-1)(\sigma_L-1)\pmod{J^3}.
\]
Therefore,
\[
\sigma_{C_b} - 1 \equiv (g-1)(\sigma_L-1)\pmod{J^3}.
\]
Thus,
\begin{equation}\label{eq:dCb-dL}
\int_{\sigma_{C_b}} I' = (g-1) \int_{\sigma_L} I'.
\end{equation}
By \eqref{eq:res-b}, \eqref{eq:res-qprime}, and \eqref{eq:dCb-dL}, we obtain
\begin{equation}\label{eq: res_rel}
\mathrm{Res}_b(\xi_{b,q}) = (g-1) \mathrm{Res}_{q'}(\xi_{b,q}).
\end{equation}
By \eqref{eq: res_sum} and \eqref{eq: res_rel}, we have \eqref{eq:res-explicit}.
\end{proof}

Lemma \ref{lem:dbdq-I-vanish}, \ref{lem:pairing-RHS-decomp}, and \ref{lem:res-explicit} show that
\begin{equation}\label{eq:pairing-RHS-modJ4}
\left\langle I, \text{RHS of \eqref{eq: modJ^4}}\right\rangle = (2g-2)\int_p^b \omega + 2\int_p^{q'} \omega.
\end{equation}
Comparing \eqref{eq:pairing-LHS-modJ4} with \eqref{eq:pairing-RHS-modJ4}, we obtain \eqref{eq: rec.law}. This completes the proof of Proposition \ref{prop: rec.law}.
\popQED}

\enlargethispage{3\baselineskip}

%%%%%%% Reference %%%%%%%%%%%%%%%%%%%%%%%%%%%%%

\end{document}